\documentclass[10pt]{article}

\usepackage{amsfonts}
\usepackage{euscript,amstext}
\usepackage{mathrsfs}
\usepackage{indentfirst}
\usepackage{enumerate}
\usepackage{amssymb}
\usepackage{leftidx}
\usepackage{amsmath}
\usepackage{amsthm}
\usepackage{hyperref}
\usepackage{amscd}

\newtheorem{thm}{Theorem}[section]

\newtheorem{prop}[thm]{Proposition}
\newtheorem{lem}[thm]{Lemma}
\newtheorem{cor}[thm]{Corollary}

\theoremstyle{definition}

\theoremstyle{remark}
\newtheorem{rem}[thm]{Remark}

\numberwithin{equation}{section}

\newcommand{\comment}[1]{}
\newcommand{\field}[1]{\mathbb{#1}}
\newcommand{\R}{\field{R}}
\newcommand{\Z}{\field{Z}}

\def\mF{\mathcal{F}}

\def\mE{\mathcal{E}}

\def\mS{\mathcal{S}}
\def\mM{\mathcal{M}}

\DeclareMathOperator{\rk}{rk}

\DeclareMathOperator{\ind}{ind}

\DeclareMathOperator{\spec}{Sp}

\DeclareMathOperator{\ch}{ch}

\DeclareMathOperator{\End}{End}

\title{Positive scalar curvature and the Euler class 
}
\author{Jianqing Yu\footnote{School of Mathematical Sciences,
University of Science and Technology of China,
96 Jinzhai Road, Hefei, Anhui 230026,
P. R. China. (jianqing@ustc.edu.cn)} \
and Weiping Zhang\footnote{Chern Institute of Mathematics \& LPMC, Nankai University, Tianjin 300071, P. R. China. (weiping@nankai.edu.cn)}
}
\date{\ }

\begin{document}

\maketitle

\begin{abstract}
We prove the following  generalization 
of the classical  Lichnerowicz vanishing theorem: if $F$ is an oriented flat vector bundle over a closed spin manifold $M$ such that $TM$ carries a   metric of positive scalar curvature, then
 $ \langle \widehat A(TM)\,e(F),[M] \rangle=0$, where $e(F)$ is the Euler class of $F$.
\end{abstract}



\section{Introduction}

 The classical Lichnerowicz vanishing theorem  \cite{MR0156292} states that if a closed spin manifold $M$ admits a Riemannian metric of positive scalar curvature, then  its Hirzebruch $\widehat A$-genus (cf. \cite[\S 1.6.3]{MR1864735}) vanishes: $ \widehat A(M)=0$.

Let $F$ be an oriented real flat vector bundle over a closed   manifold $M$. Let
$\widehat A(TM)$ denote the Hirzebruch $\widehat A$-class of $TM$ (cf. \cite[p. 13]{MR1864735}) and
 $e(F)$ denote the Euler class of $F$ (cf. \cite[\S 3.4]{MR1864735}).
The   purpose of this paper is to prove the following generalization of the    Lichnerowicz vanishing theorem.

\begin{thm}\label{t0.1} If $TM$ is spin and carries a metric of positive scalar curvature, then  $ \langle \widehat A(TM)\,e(F),[M] \rangle=0$.
\end{thm}

\begin{rem}\label{t0.2}
Recall that Milnor \cite{MR0095518} constructs on any oriented closed surface $\Sigma_g$ of genus $g>1$ a rank two oriented flat vector bundle $F_g$ such that $\langle e(F_g),[\Sigma_g]\rangle\neq 0$. 
Let $M$ be any closed spin manifold such that $TM$ carries a metric of positive scalar curvature, then $T(M\times\Sigma_g)$ also carries a metric of positive scalar curvature, and one gets by Theorem \ref{t0.1} that
\begin{align}\label{0.1}
0=\left\langle\widehat A\left(T\left(M\times \Sigma_g\right)\right)e\left(\pi^*F_g\right),\left[M\times \Sigma_g\right]\right\rangle= \widehat A(M )\left\langle e\left(F_g\right),\left[\Sigma_g\right]\right\rangle,
\end{align}
where    $\pi:M\times\Sigma_g\rightarrow \Sigma_g$ denotes the natural  projection,
which implies $\widehat A(M)=0$. Thus Theorem \ref{t0.1} indeed recovers the Lichnerowicz theorem.
\end{rem}

The following consequence of Theorem \ref{t0.1}, which generalizes the well-known fact that $T\Sigma_g$ ($g>1$) does not carry a metric of positive scalar curvature, is of independent interest.

\begin{cor}\label{t0.5}
Let $F$ be an oriented flat vector bundle over a closed spin manifold $M$  such that    ${\rm rk}(F) =\dim M$.  If $\langle e(F),[M]\rangle\neq 0$, then $TM$ does not carry  a metric of positive scalar curvature.
\end{cor}

Inspired by \cite{MR0461521}, we call a closed manifold $M$ a Smillie manifold if $TM$ carries a flat connection, while the Euler characteristic $\chi(M)\neq 0$. By Corollary \ref{t0.5}, one sees that  
a closed spin
Smillie manifold does not admit a Riemannian metric of positive scalar curvature.

On the other hand, Theorem \ref{t0.1} can also be reformulated as follows.

\begin{thm}\label{t0.4}
If a closed spin manifold $M$ admits a Riemannian metric of positive scalar curvature, then
for any representation $\rho:\pi_1(M)\rightarrow GL(q,\R)^+$, $q\in \mathbb{N}$, one has
\begin{equation}\label{vanishhigher}
\left\langle \widehat{A}(TM)\,f^*\left(e\left(E\pi\times_{\rho} \R^q\right)\right),[M]\right\rangle=0,
\end{equation}
where  $E\pi\times_{\rho} \R^{q}$, with $\pi=\pi_1(M)$, is the oriented real flat vector bundle over $B\pi$
associated to the representation $\rho$, and $f:M\rightarrow B\pi$ is the classifying map.
\end{thm}

Thus  our result   provides a new evidence to support    the famous
Gromov-Lawson  conjecture on the vanishing of
higher $\widehat{A}$-genera  \cite{MR569070}.

The main difficulty in proving Theorem \ref{t0.1} is that the flat connection on $F$ need not preserve any metric on $F$. Similar   difficulties have appeared in the foliation situations studied   in \cite{MR866491} and \cite{Zhangfoliation},  where one uses the Connes fibration introduced in \cite{MR866491}   to overcome such difficulties.

Our proof of Theorem \ref{t0.1} is index theoretic. On one hand, it is inspired by \cite{Zhangfoliation} and makes use of the Connes fibration, as well as    the constructions of deformed sub-Dirac operators. On the other hand, a significant technical difference with respect to what in \cite{Zhangfoliation}  is that while in \cite{Zhangfoliation}, one 
identifies the characteristic number in question with the indices of certain   nonexplicit pseudodifferential operators on a closed manifold, here we will work  intrinsically on  a  sub-manifold with boundary of the Connes fibration    and apply directly the analytic localization techniques developed in \cite{MR1188532},  \cite{MR1870658} and \cite{TZ}.

The rest of this  paper is organized as follows. In Section \ref{sec:flat+index},
we provide an index theoretic  interpretation of  $\langle \widehat A(TM)e(F),[M]\rangle$,
and  prove a simple vanishing result for it.
In Section \ref{sec:connes}, we construct the sub-Dirac operators as well as the key deformations of them on the
Connes fibration associated to the flat vector bundle $F$,   and reduce Theorem \ref{0.1} to an estimating result concerning the   Atiyah-Patodi-Singer elliptic boundary valued problem.
In Section \ref{sec:proofmain}, we
prove the above estimating result.

$\ $

\noindent\textbf{Acknowledgements.} We would like to thank   Xiaonan Ma
for helpful discussions. This work was partially supported by
the China Postdoctoral Science Foundation (No.\,2015T80654) and NNSFC (No.\,11401552 and No.\,11221091).

\section{Flat vector bundles and the twisted index}\label{sec:flat+index}

In this section, we provide an index theoretic  interpretation of  the characteristic number under consideration, and  prove a simple vanishing result.

This section is organized as follows. In Section \ref{s2.1}, we construct certain vector bundles associated to  the given flat vector bundle. In Section \ref{s2.2}, we construct a twisted Dirac operator of which the index is equal to the given characteristic number. A simple vanishing result is also established.


\subsection{Flat vector bundles and the   signature splitting}\label{s2.1}


Let $(F,\nabla^F)$ be an oriented real flat vector bundle over a closed manifold $M$.
Let  $g^F$  be a Euclidean metric on $F$.
As in \cite[(4.1)]{MR1185803}, set
\begin{equation}\label{omega}
\omega\left(F,g^F\right)=\left(g^F\right)^{-1}\nabla^F g^F.
\end{equation}
Then $F$ carries a canonical Euclidean connection (see \cite[(4.3)]{MR1185803})
\begin{align}
\nabla^{F,e}=\nabla^F+\frac{1}{2}\,\omega\left(F,g^F\right).
\end{align}
By \cite[Proposition 4.3]{MR1185803}, its curvature is given by
\begin{equation}\label{curvat}
\left(\nabla^{F, e}\right)^2=-\frac{1}{4}\left(\omega\left(F,g^F\right)\right)^2,
\end{equation}
which is usually nonzero.

Let $\Lambda(F^*)$ be the (complex) exterior algebra bundle of $F$. Then $\Lambda(F^*)$ carries
a Hermitian metric canonically induced from $g^F$ and a Hermitian connection $\nabla^{\Lambda(F^*), u}$
induced from $\nabla^{F, e}$.

For any $f\in F$, let $f^*\in F^*$ be the metric dual of $f$ with respect to $g^F$. Let $c(f)$, $\widehat c(f)$ be the Clifford actions on $\Lambda(F^*)$ defined by
\begin{equation}\label{cliff}
c(f)=f^*\wedge-i_f,\quad \widehat{c}(f)=f^*\wedge+i_f,
\end{equation}
where $f^*\wedge$ and $i_f$ are the exterior and interior multiplications
by $f^*$ and $f$.

Let $\{f_\mu\}_{\mu=1}^{\rk(F)}$ be an orthonormal basis of $(F,g^{F})$, and
$\{f^\mu\}_{\mu=1}^{\rk(F)}$ be its dual basis.
Then by \cite[(1.26)]{MR2273508}, the curvature $(\nabla^{\Lambda(F^*),u})^2$ of $\nabla^{\Lambda(F^*),u}$ can be computed as follows,
\begin{equation}\label{curvat2}
\left(\nabla^{\Lambda\left(F^*\right),u}\right)^2
=\sum_{\mu,\,\nu=1}^{\rk(F)}\left\langle\left (\nabla^{F,e}\right)^2f_\mu,f_\nu\right\rangle
f^\nu\wedge i_{f_\mu}.
\end{equation}
From \eqref{curvat}-\eqref{curvat2}, we deduce that
\begin{align}\label{curvat3}
\left(\nabla^{\Lambda\left(F^*\right),u}\right)^2
 =-\frac{1}{16}\sum_{\mu,\,\nu=1}^{\rk(F)}
\left\langle f_\mu, \omega\left(F,g^F\right) ^2f_\nu\right\rangle
\left(\widehat{c}\left(f_\mu\right)\widehat{c}\left(f_\nu\right)-c\left(f_\mu\right)c\left(f_\nu\right)\right).
\end{align}

 On the other hand, when $\rk(F)$ is even, set
\begin{equation}\label{tauF}
\tau\left(F,g^F\right)=\left(\frac{1}{\sqrt{-1}}\right)^{\frac{\rk (F)}{2}}c\left(f_1\right)\,\cdots\,c\left(f_{\rk(F)}\right).
\end{equation}
Then $(\tau(F,g^F))^2={\rm Id}\big|_{\Lambda(F^*)}$, and $\tau(F,g^F)$ induces a $\Z_2$-graded splitting
\begin{equation}\label{Z2F}
\Lambda\left(F^*\right)=\Lambda_{+}\left(F^*\right)\oplus \Lambda_{-}\left(F^*\right),
\end{equation}
where
\begin{equation}\label{lambda}
\Lambda_{\pm}\left(F^*\right)=\left\{\alpha\in \Lambda\left(F^*\right)\,\big|\, \tau\left(F,g^F\right)\alpha=\pm\alpha\right\}.
\end{equation}
Moreover, $c(f)$ exchanges $\Lambda_{\pm}(F^*)$ for any $f\in F$.

\subsection{Dirac operators and an easy vanishing result}\label{s2.2}


We
 assume   that $M$ is  spin and  of  even dimension, and that
  $\rk(F)$ is even.
Let $g^{TM}$ be a Riemannian metric on $M$, and
$\nabla^{TM}$ be the associated Levi-Civita connection.
Let $S(TM)$ be the Hermitian bundle
of spinors associated to $(TM,g^{TM})$ with the $\mathbb{Z}_2$-graded splitting
\begin{equation}\label{Z2S}
S(TM)=S_{+}(TM)\oplus S_{-}(TM).
\end{equation}
Then $\nabla^{TM}$ induces naturally a Hermitian connection $\nabla^{S(TM)}$ on $S(TM)$ preserving the $\mathbb{Z}_2$-grading.

From the $\Z_2$-graded vector bundles in \eqref{Z2F} and \eqref{Z2S}, we form the following $\Z_2$-graded tensor product (see \cite[p. 11]{MR1031992})
\begin{equation}\label{Z2SF}
S(TM)\widehat{\otimes}\Lambda\left(F^*\right)=\left(S(TM)\widehat{\otimes}\Lambda\left(F^*\right)\right)
_+\oplus\left(S(TM)\widehat{\otimes}\Lambda\left(F^*\right)\right)_-,
\end{equation}
which carries the $\Z_2$-graded tensor product connection
\begin{equation}
\nabla^{u}=\nabla^{S(TM)}\otimes {\rm Id}|_{ \Lambda(F^*)} + {\rm Id}_{S(TM) } \otimes\nabla^{\Lambda(F^*),u}.
\end{equation}

For $e\in TM$, we denote by $c(e)$ the Clifford action of $e$ on $S(TM)$. Then it extends to an action
$c(e)\otimes {\rm Id}\big|_{\Lambda(F^*)}$ on $S(TM)\widehat{\otimes}\Lambda(F^*)$,
which we still denote by $c(e)$.

Take an oriented orthonormal basis $\{e_i\}_{i=1}^{\dim M}$ of $(TM,g^{TM})$.
Let
\begin{align}\label{Mdirac}
D^M=\sum_{i=1}^{\dim M}c\left(e_i\right)\nabla^{u}_{e_i}:
 \Gamma\left(M,S(TM)\widehat{\otimes}\Lambda\left(F^*\right)\right)\longrightarrow
\Gamma\left(M,S(TM)\widehat{\otimes}\Lambda\left(F^*\right)\right)
\end{align}
be the corresponding  twisted Dirac operator, and denote
\begin{equation}\label{Mdiracpm}
D^M_\pm=D^M\big|_{\Gamma\left(M,(S(TM)\widehat{\otimes}\Lambda(F^*))_\pm\right)}\ .
\end{equation}

Since $(F,\nabla^F)$ is flat, by   the Atiyah-Singer index theorem \cite{MR0236950}, we get
\begin{align}\label{asindex}
\ind\left(D^M_+\right)&=\left\langle \widehat A(TM)\ch
\left(\Lambda_+\left(F^*\right)-\Lambda_-\left(F^*\right)\right),[M]\right\rangle
\notag\\
&=2^{\frac{\rk(F)}{2}}\left\langle \widehat{A}(TM)\,e(F),[M]\right\rangle.
\end{align}


 Let
$\Delta^{u}=\sum_{i=1}^{\dim M}
(\nabla^{u}_{e_i}\nabla^{u}_{e_i}-\nabla^{u}_{\nabla^{TM}_{e_i}e_i}
)
$  
be the Bochner Laplacian. Let $k^{TM}$ denote the scalar curvature of $(TM,g^{TM})$. We have the following  standard Lichnerowicz formula \cite{MR0156292},
\begin{align}
\label{D2}
\left(D^M\right)^2=-\Delta^{u}+\frac{k^{TM}}{4}+
\frac{1}{2}\sum_{i,\,j=1}^{\dim M}c\left(e_i\right)c\left(e_j\right)\left(\nabla^{\Lambda\left(F^*\right),u}\right)^2 \left(e_i,e_j\right).
\end{align}

From (\ref{asindex}) and (\ref{D2}), one obtains the following easy vanishing result.

\begin{prop}\label{t2.1}
If there holds over $M$ that
\begin{align}\label{001}
\frac{k^{TM}}{4}+
\frac{1}{2}\sum_{i,\,j=1}^{\dim M}c\left(e_i\right)c\left(e_j\right)\left(\nabla^{\Lambda\left(F^*\right),u}\right)^2 \left(e_i,e_j\right)>0,
\end{align}
then one has $\langle \widehat{A}(TM)\,e(F),[M] \rangle=0$.
\end{prop}


In the  next two sections, we will   eliminate the summation term in (\ref{001}).

\section{Connes fibration and sub-Dirac operators}\label{sec:connes}

In this section, we reduce the proof of Theorem \ref{t0.1} to an estimating result of certain Atiyah-Patodi-Singer elliptic boundary valued problems for deformed sub-Dirac  operators constructed  on the Connes fibration associated to a given flat vector bundle.

This section is organized as follows. In Section \ref{s3.1}, we present the  construction of  the Connes fibration associated to a flat vector bundle as well as certain basic properties of the Connes fibration. In Section \ref{sec:subdirac}, we construct the needed sub-Dirac operator on the Connes fibration. In Section \ref{s3.3}, we study the induced sub-Dirac operator on the boundary. In Section \ref{s3.4}, we introduce certain deformations of the sub-Dirac operators and reduce Theorem \ref{t0.1} to an estimating result.

\subsection{Connes fibration associated with a flat bundle}\label{s3.1}

Let $(M,g^{TM})$ be a closed Riemannian manifold,
and $(F,\nabla^F)$ an oriented real flat vector bundle over $M$.
Following \cite[\S5]{MR866491} (cf. \cite[\S 2.1]{Zhangfoliation}), let $\pi:\mM\rightarrow M$ be the Connes
fibration over $M$ such that for any $x\in M$, $\mM_{x}$ is the space of
Euclidean metrics on the vector space $F_x$.
Let $\mE^{\perp}=T^V\mM$ denote the vertical tangent bundle of this fibration.
Then it carries a naturally induced metric $g^{\mE^{\perp}}$ such that each $\mM_{x}=\pi^{-1}(x)$, $x\in M$, is of nonpositive sectional curvature. In particular, any two points $p_1,\,p_2\in\mM_x$ can be joined by a
unique geodesic in $\mM_x$ (cf. \cite{MR1834454}). Let $d^{\mM_x}(p_1,p_2)$ denote the length of this geodesic .

By using the flat connection $\nabla^F$, one lifts $TM$ to an integrable  horizontal subbundle $\mE=T^H\mM$ of $T\mM$ so that we have a canonical splitting
$T\mM=\mE\oplus\mE^{\perp}.$

Set $\mF=\pi^*F$. Then there is a canonical Euclidean metric $g^\mF$ on $\mF$ defined as follows: by construction,   any $p\in \mM$ determines a Euclidean metric on $F_{\pi(p)}$, which in turn determines a metric on $\mF_p\simeq \pi^*F_{\pi(p)}$.

\begin{lem}\label{t3.1} 1). The Bott connection on $(\mE^\perp,g^{\mE^\perp})$ is leafwise Euclidean.

2). There exists a canonical Euclidean connection on $(\mF,g^\mF)$ such that for any $X,\, Y\in \Gamma(\mM,\mE)$, one has
\begin{align}\label{mF_horizonflat}
\left(\nabla^{\mF}\right)^2(X,Y)=0.
\end{align}
\end{lem}
\begin{proof}
Let $\widehat F$ denote the total space of the flat vector bundle $\pi_F: F\rightarrow M$. Then    $TM$ lifts to an integrable subbundle $T^H\widehat F$ of $T\widehat F$ such that $T^H\widehat F|_{M}=TM$, and that $(T\widehat F/T^H\widehat F)|_M\simeq F$.

Following  \cite[\S5]{MR866491} and \cite[\S 2.1]{Zhangfoliation}, let $\widehat \pi :\widehat \mF\rightarrow \widehat F$ be the Connes fibration such that for any $x\in\widehat F$, $\widehat \mF_x=\widehat\pi^{-1}(x)$ is the space of Euclidean metrics on $T\widehat F_x/T^H\widehat F_x$.
Then one verifies that
\begin{align}\label{3.1}
\mM\simeq \widehat\pi^{-1}(M).
\end{align}

By restricting \cite[Lemma 1.5]{Zhangfoliation} from $\widehat \mF$ to $\mM$, one gets Lemma \ref{t3.1}.  We leave the details to the interested reader.
\end{proof}

 {

Let $g^{\mE}=\pi^*g^{TM}$ be the pullback Euclidean metric on $\mE$. Let $g^{T\mM}$ be the Riemannian metric on $\mM$ given by the orthogonal splitting
\begin{align}\label{TmM_metric}
T\mM=\mE\oplus \mE^\perp,\ \ \ g^{T\mM}=g^{\mE}\oplus g^{\mE^\perp}.
\end{align}
Let $p$ and $p^\perp$ be the orthogonal projections from $T\mM$ to $\mE$
and $\mE^\perp$.
Let $\nabla^{T\mM}$ be the Levi-Civita connection of $g^{T\mM}$. Set
\begin{align}\label{}
\nabla^{\mE}=p\nabla^{T\mM}p,
\quad
\nabla^{\mE^\perp}=p^\perp\nabla^{T\mM}p^\perp.
\end{align}
Then $\nabla^{\mE^\perp}$ does not depend on $g^\mE$. Moreover, by Lemma \ref{t3.1},  one has
\begin{align}\label{mEperp_horizonflat}
\left(\nabla^{\mE^\perp}\right)^2(X,Y)=0,\ \text{for any}\ X, Y\in \Gamma(\mM,\mE).
\end{align}

}

Take a Euclidean metric $g^F$ on $F$, which amounts to taking an embedded
section $\jmath:M\hookrightarrow \mM$ of $M$ into the Connes
fibration  $\pi:\mM\rightarrow M$.
Then we have the canonical inclusion  $\jmath(M)\subset \mM$.

For any $p\in \mM\setminus \jmath(M)$,  we connect
$p$ and $\jmath(\pi(p))\in \jmath(M)$ by the unique geodesic in $\mM_{\pi(p)}$.
Let $\sigma(p)\in\mE^\perp|_p$ denote the unit vector tangent to this geodesic.
Set $\rho(p)=d^{\mM_{\pi(p)}}(p,\jmath(\pi(p)))$.

By (\ref{3.1}) and \cite[Lemma~2.1]{Zhangfoliation}, we have the following estimating result.

\begin{lem}\label{lem:esti}
There exists  $C>0$, which depends only on the
embedding $\jmath:M\hookrightarrow \mM$, such that for any $X\in\Gamma(\mM,\mE)$
with $|X|\leqslant 1$, the following pointwise 
inequality holds on $\mM$,
\begin{equation}
\left| \nabla^{\mE^\perp}_X(\rho\sigma)\right|\leqslant C.
\end{equation}
\end{lem}

\subsection{Sub-Dirac operators on the Connes fibration}\label{sec:subdirac}

Without loss of generality, we can and we will assume that both $\dim M$ and ${\rm rk}(F)$ are divisible  by $4$.
Then both  $\dim\mM$ and ${\rm rk}(\mE^\perp)$ are   even. By passing to a double covering if necessary, we also assume that $\mE^\perp$ is oriented.

We   assume from now on that $M$ is spin, then  $\mE=\pi^{*}(TM)$ is spin and carries a
naturally induced spin structure. Let
\begin{equation}\label{Z2SE}
S(\mE)=S_{+}(\mE)\oplus S_{-}(\mE)
\end{equation}
be the $\mathbb{Z}_2$-graded Hermitian bundle
of spinors associated  to $(\mE,g^\mE)$.
Then $\nabla^{\mE}$ induces naturally a Hermitian connection
$\nabla^{S(\mE)}$ on $S(\mE)$ preserving the $\mathbb{Z}_2$-grading.
For $e\in\mE$, let $c(e)$ denote the Clifford action of $e$ on $S(\mE)$, which exchanges $S_\pm(\mE)$.

Let $\Lambda(\mE^{\perp,*})$ be the (complex) exterior algebra bundle
of $\mE^\perp$ with the $\mathbb{Z}_2$-graded splitting
\begin{equation}\label{Z2LamEperp}
\Lambda\left(\mE^{\perp,*}\right)=\Lambda^{\rm even}\left(\mE^{\perp,*}\right)\oplus \Lambda^{\rm odd}\left(\mE^{\perp,*}\right).
\end{equation}
Then $\Lambda(\mE^{\perp,*})$ carries
a Hermitian metric canonically induced from $g^{\mE^{\perp}}$ and a Hermitian connection $\nabla^{\Lambda(\mE^{\perp,*})}$
induced from $\nabla^{\mE^\perp}$.
For $h\in\mE^{\perp}$, let $c(h)$ and $\widehat{c}(h)$ be
the actions of $h$ on $\Lambda(\mE^{\perp,*})$ defined as in \eqref{cliff}.


Let $\Lambda(\mF^*)$ be the (complex) exterior algebra bundle of $\mF$.
The Euclidean connection $\nabla^{\mF}$ on $\mF$
naturally induces a connection $\nabla^{\Lambda(\mF^*)}$
on $\Lambda(\mF^*)$, which
preserves the metric on $\Lambda(\mF^*)$ induced by $g^{\mF}$.
We denote by $c(\cdot)$ and $\widehat{c}(\cdot)$ the actions of
$\mF$ on $\Lambda(\mF^*)$ defined as in \eqref{cliff}. Let
\begin{equation}\label{Z2LamF}
\Lambda(\mF^*)=\Lambda_{+}(\mF^*)\oplus \Lambda_{-}(\mF^*)
\end{equation}
be the $\mathbb{Z}_2$-graded splitting of $\Lambda(\mF^*)$ determined as in \eqref{tauF}-\eqref{lambda}.

Using the $\Z_2$-graded vector bundles in \eqref{Z2SE}-\eqref{Z2LamF}, we form the following $\Z_2$-graded tensor product (see \cite[p. 11]{MR1031992})
\begin{align}\label{totalZ2grad}
&S(\mE)\widehat{\otimes}\Lambda\left(\mE^{\perp,*}\right)\widehat{\otimes}\Lambda\left(\mF^*\right)
\notag\\
&=\left(S(\mE)\widehat{\otimes}\Lambda\left(\mE^{\perp,*}\right)\widehat{\otimes}\Lambda\left(\mF^*\right)
\right)_+
\oplus\left(S(\mE)\widehat{\otimes}\Lambda\left(\mE^{\perp,*}\right)\widehat{\otimes}\Lambda
\left(\mF^*\right)\right)_-\ ,
\end{align}
which carries the tensor product connection
$\nabla^{S(\mE)\widehat{\otimes}\Lambda(\mE^{\perp,*})\widehat{\otimes}\Lambda(\mF^*)}$ induced by
$\nabla^{S(\mE)}$, $\nabla^{\Lambda(\mE^{\perp,*})}$ and $\nabla^{\Lambda(\mF^*)}$.

The action of $\mE$ on $S(\mE)$ and that of
$\mE^\perp$ on $\Lambda(\mE^{\perp,*})$
as well as that of $\mF$ on $\Lambda(\mF^*)$ extend
to actions on $S(\mE)\widehat{\otimes}\Lambda(\mE^{\perp,*})\widehat{\otimes}\Lambda(\mF^*)$
in an obvious way. We still use the same notation to denote these extended actions.

Let
$
\{h_i\}_{i=1}^{\dim M}$
 (resp. $
\{h_j\}_{j=\dim M+1}^{\dim\mM}
$)  
be an oriented orthonormal basis of $(\mE,g^{\mE})$ (resp. $(\mE^\perp,g^{\mE^\perp})$).
Let $\mS$ be the $\End(T\mM)$-valued one-form on $\mM$ defined by
\begin{equation}
\mS(\cdot)=\nabla_{\cdot}^{T\mM}-\left(\nabla_{\cdot}^{\mE}\oplus \nabla_{\cdot}^{\mE^\perp}\right).
\end{equation}
Following \cite{MR1850748}  and \cite[(1.60)]{Zhangfoliation}, we define a Hermitian connection
\begin{equation}\label{tensorconnection}
\widehat{\nabla}_{\cdot}=
\nabla_{\cdot}^{S(\mE)\widehat{\otimes}\Lambda(\mE^{\perp,*})\widehat{\otimes}\Lambda(\mF^*)}
+\frac{1}{4}\sum_{i,\,j=1}^{\dim\mM}\left\langle \mS(\cdot)h_i,h_j\right\rangle c\left(h_i\right)c\left(h_j\right)
\end{equation}
on $S(\mE)\widehat{\otimes}\Lambda(\mE^{\perp,*})\widehat{\otimes}\Lambda(\mF^*)$.

As in \cite{MR1850748} and \cite[(1.61)]{Zhangfoliation}, let
\begin{align}\label{subdirac}
D^{\mM}&=\sum_{i=1}^{\dim\mM}c\left(h_i\right)\widehat{\nabla}_{h_i}:
\Gamma\left(\mM,
S(\mE)\widehat{\otimes}\Lambda\left(\mE^{\perp,*}\right)\widehat{\otimes}\Lambda\left(\mF^*\right)\right)
\notag\\
&\hspace{10em}
\longrightarrow\Gamma\left(\mM,
S(\mE)\widehat{\otimes}\Lambda\left(\mE^{\perp,*}\right)\widehat{\otimes}\Lambda\left(\mF^*\right)\right)
\end{align}
be the sub-Dirac
operator with respect to the spinor bundle $S(\mE)$.
Then $D^{\mM}$ is a formally self-adjoint first order elliptic differential operator,
which exchanges the $\mathbb{Z}_2$-grading in \eqref{totalZ2grad}.
Moreover, as indicated  in  \cite[Remark~1.8]{Zhangfoliation}, $D^\mM$ can locally be viewed as a twisted Dirac operator.

Let
$\Delta^{\mM}=\sum_{i=1}^{\dim\mM} (\widehat{\nabla}_{h_i}
\widehat{\nabla}_{h_i}-\widehat{\nabla}_{\nabla^{T\mM}_{h_i}h_i})$ be the Bochner Laplacian,
and $k^{T\mM}$ be the scalar curvature of $(T\mM,g^{T\mM})$.
 We have the following   Lichnerowicz formula,
\begin{multline}\label{Lichnerowicz}
\left(D^{\mM}\right)^2
=-\Delta^{\mM}+\frac{k^{T\mM}}{4}
\\
+\frac{1}{8}\sum_{k,\,l=\dim M+1}^{\dim\mM}\sum_{i,\,j=1}^{\dim \mM}\left\langle h_k,\left(\nabla^{\mE^{\perp}}\right)^2\left(h_i,h_j\right)h_l\right\rangle
\cdot c\left(h_i\right)c\left(h_j\right)\widehat{c}\left(h_k\right)\widehat{c}\left(h_l\right)
\\
+\frac{1}{8}\sum_{\mu,\,\nu=1}^{\rk(\mF)}\sum_{i,\,j=1}^{\dim\mM}\left\langle f_\mu,
\left(\nabla^{\mF}\right)^2\left(h_i,h_j\right)f_\nu \right\rangle
\cdot c\left(h_i\right)c\left(h_j\right)\left(\widehat{c}\left(f_\mu\right)\widehat{c}\left(f_\nu\right)-c\left(f_\mu\right)c(f_\nu)\right),
\end{multline}
where $\{f_\mu\}_{\mu=1}^{\rk \mF}$ is an orthonormal basis of $(\mF,g^{\mF})$.
\subsection{Induced sub-Dirac operators  on the boundary}\label{s3.3}

For any $R>0$, denote
$
\mM_R=\{p\in\mM\,|\,\rho(p)\leqslant R\}.
$
Then $\mM_R$ is a compact smooth   manifold with boundary $\partial\mM_R$.

We follow the convention as in   \cite{MR1252030}.

Let $\epsilon_R>0$ be a sufficiently small positive number.
We use the inward geodesic flow to identify a neighborhood of $\partial\mM_R$ with the collar
$\partial\mM_R\times [0,\epsilon_R)$. Let $e_{\dim\mM}$ be the inward unit normal
vector field to $T\partial\mM_R$ so that $e_1,\cdots,e_{\dim\mM}$ is an oriented orthonormal basis of
$T\mM|_{\partial\mM_R}$. Then using parallel transport with respect to $\nabla^{T\mM}$ along
the unit speed geodesics perpendicular to $\partial\mM_R$,  $e_1,\cdots,e_{\dim\mM}$
forms an oriented orthonormal basis of $T\mM$ over $\partial\mM_R\times [0,\epsilon_R)$.

For   $1\leqslant i,\,j\leqslant \dim \mM-1$, let
$\pi_{ij}=\langle\nabla^{T\mM}_{e_i}e_j,e_{\dim\mM}
\rangle|_{\partial\mM_R}$
be the second fundamental form of the isometric
embedding $i_{\partial\mM_R}:\partial \mM_R\hookrightarrow \mM_R$.
Let
\begin{align*}
&D^{\partial \mM_R}:
 \Gamma\left(\left.\partial\mM_R,\left(S(\mE)\widehat{\otimes}\Lambda\left(\mE^{\perp,*}\right)
\widehat{\otimes}\Lambda\left(\mF^{*}\right)\right)\right|_{\partial\mM_R}\right)
\\
&\hspace{6em}
\longrightarrow
 \Gamma\left(\left.\partial\mM_R,\left(S(\mE)\widehat{\otimes}\Lambda\left(\mE^{\perp,*}\right)
\widehat{\otimes}\Lambda\left(\mF^{*}\right)\right)\right|_{\partial\mM_R}\right)
\end{align*}
be the differential operator on ${\partial\mM_R}$ defined by
\begin{align}\label{bd_subdirac}
D^{\partial \mM_R}
=-\sum_{i=1}^{\dim\mM-1}c\left(e_{\dim\mM}\right)c\left(e_i\right)\widehat{\nabla}_{e_i}
+\frac{1}{2}\sum_{i=1}^{\dim\mM-1}\pi_{ii}\ .
\end{align}
By \cite[Lemmas {2.1 and 2.2}]{MR1252030}, $D^{\partial \mM_R}$ is a formally self-adjoint
first order elliptic differential operator intrinsically defined on $\partial\mM_R$.
Also, it preserves the $\Z_2$-grading of
$(S(\mE)\widehat{\otimes}\Lambda(\mE^{\perp,*})
\widehat{\otimes}\Lambda(\mF^{*}))|_{\partial\mM_R}$ induced by \eqref{totalZ2grad}.

\subsection{Deformations of sub-Dirac operators and their indices}\label{s3.4}

Let $\psi: [0,1]\rightarrow [0,1]$ be a smooth function such that
$\psi(t)=0$ for $0 \leqslant t \leqslant\frac{2}{3}$, while
$\psi(t)=1$ for $\frac{3}{4}\leqslant t \leqslant 1$.
For any $0<\varepsilon\leqslant 1$ and $R>0$, let
\begin{align}\label{cut_adia_metric}
{g}_{\varepsilon,R}^{T\mM}
=\left(1-\psi\left(\frac{\rho}{R}\right)\right)g_\varepsilon^{T\mM}
+\psi\left(\frac{\rho}{R}\right)g^{T\mM}
\end{align}
be the Riemannian metric on $\mM_R$, where ${g}_{\varepsilon}^{T\mM}$ is the Riemannian metric on $\mM$ defined by the orthogonal splitting (cf. \cite[(1.11)]{Zhangfoliation})
\begin{align}\label{3.2}
T\mM=\mE\oplus\mE^\perp,\ \ \ {g}_{\varepsilon}^{T\mM}
= \varepsilon^2g^{\mE}\oplus  g^{\mE^{\perp}}.
\end{align}
Then
\begin{equation}\label{cut_metric_inbd}
{g}_{\varepsilon,R}^{T\mM}={g}_{\varepsilon}^{T\mM}\
\text{over}\ \mM_{\frac{2}{3}R}\,, \
\text{while}\
{g}_{\varepsilon,R}^{T\mM}={g}^{T\mM}\
\text{over}\ \mM_R\setminus \mM_{\frac{3}{4}R}\,.
\end{equation}

In what follows, we will use the subscripts (or superscripts)\,
\textquotedblleft\,$\varepsilon$, $R$\,\textquotedblright\,to
decorate the geometric objects with respect to ${g}_{\varepsilon,R}^{T\mM}$.

Let $k_{\varepsilon,R}$ be the scalar curvature of ${g}_{\varepsilon,R}^{T\mM}$.
Then by Lemma \ref{t3.1} and \eqref{cut_metric_inbd}  (cf. \cite[Proposition~1.4]{Zhangfoliation}), one has
over $\mM_{\frac{2}{3}R}$ that
\begin{align}\label{scalarcurv}
k_{\varepsilon,R}=\frac{\pi^*k^{TM}}{\varepsilon^2}+O_R(1),
\end{align}
where $k^{TM}$ is the scalar curvature of $(TM,g^{TM})$, and
by $O_R(\cdot)$ we mean that the estimating constant might depend on $R$.

As in Subsection \ref{sec:subdirac}, we construct the sub-Dirac operator
\begin{align*}\label{}
&D^{\mM_R}_{\varepsilon}:
\Gamma\left( \mM_R,
S_{\varepsilon,R}(\mE)\widehat{\otimes}\Lambda\left(\mE^{\perp,*}\right)
\widehat{\otimes}\Lambda\left(\mF^*\right)\right)
\notag\\
&\hspace{8em}
\longrightarrow
 \Gamma\left( \mM_R,
S_{\varepsilon,R}(\mE)\widehat{\otimes}\Lambda\left(\mE^{\perp,*}\right)
\widehat{\otimes}\Lambda\left(\mF^*\right)\right)
\end{align*}
on $\mM_R$ associated with $g^{T\mM}_{\varepsilon,R}$, which is given by
\begin{align}\label{subdirac_epsilonR}
D^{\mM_R}_{\varepsilon}
=
\sum_{j=1}^{\dim M}c_{\varepsilon,R}(h_j)\widehat{\nabla}^{\varepsilon,R}_{h_j}
+\sum_{j=\dim M+1}^{\dim \mM}c(h_j)\widehat{\nabla}^{\varepsilon,R}_{h_j},
\end{align}
where $\{h_j\}_{j=1}^{\dim M}$ is an orthonormal basis of $(\mE,\left(\varepsilon^2(1- \psi(\frac{\rho}{R}))+ \psi(\frac{\rho}{R})\right)g^\mE$),
and
 $\widehat{\nabla}^{\varepsilon,R}$  denotes the connection  on $S_{\varepsilon,R}(\mE)\widehat{\otimes}\Lambda(\mE^{\perp,*})
\widehat{\otimes}\Lambda(\mF^*)$ associated with
${g}_{\varepsilon,R}^{T\mM}$ as in \eqref{tensorconnection}.
In particular, in view of \cite[Remark~1.8]{Zhangfoliation}, one deduces that for any   $V\in\Gamma(\mM_R,\mE^\perp)$,
\begin{align}\label{[conn,hatc]}
\left[\widehat{\nabla}^{\varepsilon,R},\widehat{c}(V)\right]
=\widehat{c}\left(\nabla^{\mE^\perp}V\right).
\end{align}

Inspired by \cite[(2.21) and Remark 2.6]{Zhangfoliation}, we introduce the following key deformation of the sub-Dirac operator $D^{\mM_R}_\varepsilon$, 
\begin{align}\label{sbdirac_deform}
&D^{\mM_R}_{\varepsilon,R}=D^{\mM_R}_{\varepsilon}+
\frac{\widehat{c}(\rho\sigma)}{\varepsilon R}:
 \Gamma\left(\mM_R,
S_{\varepsilon,R}(\mE)\widehat{\otimes}\Lambda\left(\mE^{\perp,*}\right)
\widehat{\otimes}\Lambda\left(\mF^*\right)\right)
\notag\\
&\hspace{8em}
\longrightarrow
 \Gamma\left(\mM_R,
S_{\varepsilon,R}(\mE)\widehat{\otimes}\Lambda\left(\mE^{\perp,*}\right)
\widehat{\otimes}\Lambda\left(\mF^*\right)\right)
.
\end{align}

Recall that $D^{\partial \mM_R}$ is defined in \eqref{bd_subdirac}.
Set 
\begin{align}\label{bdop_deform}
&D^{\partial \mM_R}_{\varepsilon}
=D^{\partial \mM_R}-\frac{1}{\varepsilon}c\left(e_{\dim\mM}\right)
\widehat{c}(\sigma):
 \Gamma\left(\partial\mM_R,\left.\left(S(\mE)\widehat{\otimes}\Lambda\left(\mE^{\perp,*}\right)
\widehat{\otimes}\Lambda\left(\mF^{*}\right)\right)
\right|_{\partial\mM_R}\right)
\notag\\
&\hspace{8em}
\longrightarrow
 \Gamma\left(\partial\mM_R,\left.\left(S(\mE)\widehat{\otimes}\Lambda\left(\mE^{\perp,*}\right)
\widehat{\otimes}\Lambda\left(\mF^{*}\right)\right)\right|_{\partial\mM_R}\right)\ .
\end{align}
Then
$D^{\partial \mM_R}_{\varepsilon}$ is   the induced boundary
operator of $D^{\mM_R}_{\varepsilon,R}$. Write
\begin{equation*}
D^{\mM_R}_{\varepsilon,R,\pm}=\left.D^{\mM_R}_{\varepsilon,R}
\right|_{\Gamma(\mM_R,(S_{\varepsilon,R}(\mE)\widehat{\otimes}\Lambda(\mE^{\perp,*})
\widehat{\otimes}\Lambda(\mF^*))_\pm)}\ ,
\end{equation*}
\begin{align*}
D^{\partial \mM_R}_{\varepsilon,\pm}=\left.D^{\partial \mM_R}_{\varepsilon}
\right|_{
 \Gamma(\partial\mM_R,(S(\mE)\widehat{\otimes}\Lambda\left(\mE^{\perp,*}\right)
\widehat{\otimes}\Lambda\left(\mF^{*}\right))_{\pm}|_{\partial\mM_R})
}.
\end{align*}

For any $\lambda\in \spec\{D^{\partial \mM_R}_{\varepsilon}\}$,
the spectrum of $D^{\partial \mM_R}_{\varepsilon}$,
let $E_\lambda$ be the eigenspace corresponding to $\lambda$. For any $b\in\R$, denote by
$P_{\geqslant b,\varepsilon,R}$ and $P_{> b,\varepsilon,R}$ the orthogonal projections from the $L^2$-completion of
\begin{equation*}
 \Gamma\left(\partial\mM_R,\left.\left(S(\mE)\widehat{\otimes}\Lambda\left(\mE^{\perp,*}\right)
\widehat{\otimes}\Lambda\left(\mF^{*}\right)\right)\right|_{\partial\mM_R}\right)
\end{equation*}
onto $\oplus_{\lambda\geqslant b}E_\lambda$
and  $\oplus_{\lambda>b}E_\lambda$, respectively. Let$P_{\geqslant b,\varepsilon,R,\pm}$
and $P_{> b,\varepsilon,R,\pm}$ be
the restrictions of  $P_{\geqslant b,\varepsilon,R}$ and $P_{> b,\varepsilon,R}$
on the $L^2$-completions of
\begin{equation*}
 \Gamma\left(\partial\mM_R,\left.\left(S(\mE)\widehat{\otimes}\Lambda\left(\mE^{\perp,*}\right)
\widehat{\otimes}\Lambda\left(\mF^{*}\right)\right)_{\pm}\right|_{\partial\mM_R}\right).
\end{equation*}

One verifies  (cf. \cite{MR0397797} and \cite{MR1252030}) that the Atiyah-Patodi-Singer boundary  valued problems $(D^{\mM_R}_{\varepsilon,R,+},
P_{\geqslant 0,\varepsilon,R,+})$ and
$(D^{\mM_R}_{\varepsilon,R,-},P_{> 0,\varepsilon,R,-})$
  are elliptic, and that $(D^{\mM_R}_{\varepsilon,R,+},
P_{\geqslant 0,\varepsilon,R,+})$ is the adjoint of $(D^{\mM_R}_{\varepsilon,R,-},P_{>0,\varepsilon,R,-})
$. Set
\begin{align}
\ind \left(D^{\mM_R}_{\varepsilon,R,+},
P_{\geqslant 0,\varepsilon,R,+}\right)
&=\dim \ker \left(D^{\mM_R}_{\varepsilon,R,+},
P_{\geqslant 0,\varepsilon,R,+}\right)
\notag\\
&\hspace{3em}
-\dim \ker \left(D^{\mM_R}_{\varepsilon,R,-},
P_{>0,\varepsilon,R,-}\right).
\end{align}

Let $\|\cdot\|_{\partial \mM_R}$
denote the $L^2$-norm with respect to the volume element $dv_{\partial\mM_R}$
of  $(\partial\mM_R,g^{T\mM}|_{\partial\mM_R})$.  Let $e_i$, $1\leqslant i\leqslant \dim\mM-1$, be an orthonormal basis of $(\partial\mM,g^{T\mM}|_{\partial\mM})$.

\begin{prop}\label{prop:bdop_inverti}
For any (fixed) $R>0$, there exists
$\varepsilon_0>0$ (which may depend on $R$) such that for any $0<\varepsilon\leqslant \varepsilon_0$,
one has
\begin{align}\label{bdop_esti}
\left\|D^{\partial \mM_R}_{\varepsilon} s\right\|^2_{\partial\mM_R}
\geqslant \frac{1}{2}\sum_{i=1}^{\dim\mM} \left\| \widehat\nabla _{e_i}s\right\|^2_{\partial \mM_R}+
\frac{1}{4\varepsilon^2}\| s\|^2_{\partial\mM_R}\ ,
\end{align}
for any
$s\in \Gamma\left(\partial\mM_R,\left.\left(S(\mE)\widehat{\otimes}\Lambda\left(\mE^{\perp,*}\right)
\widehat{\otimes}\Lambda\left(\mF^{*}\right)\right)\right|_{\partial\mM_R}\right)$.
\end{prop}
\begin{proof}
From \eqref{bdop_deform}, one deduces that
\begin{align}\label{3.3}
 \left(D^{\partial \mM_R}_{\varepsilon} \right)^2  = \left(D^{\partial \mM_R} \right)^2 -\frac{ 1}{\varepsilon} \left[D^{\partial \mM_R} ,c\left(e_{\dim\mM}\right)\widehat c(\sigma)\right]+\frac{1}{\varepsilon^2}.
\end{align}

It is easy to see that  $[D^{\partial \mM_R} ,c\left(e_{\dim\mM}\right)\widehat c(\sigma)]$ is of zeroth order. On the other hand, by the Lichnerowicz formula one deduces that
\begin{align}\label{3.4}
\left\langle \left(D^{\partial \mM_R} \right)^2s,s\right\rangle_{\partial\mM_R}\geqslant \frac{1}{2}
\sum_{i=1}^{\dim\mM} \left\| \widehat\nabla  _{e_i}s\right\|^2_{\partial \mM_R}
+O_R(1) \,\| s\|^2_{\partial\mM_R}\,.
\end{align}

From (\ref{3.3}) and (\ref{3.4}), one gets (\ref{bdop_esti}).
\end{proof}

Recall that $D^M$ is introduced in \eqref{Mdirac} and \eqref{Mdiracpm}.

Since $D^{\mM_R}_\varepsilon=D^{\mM_R}$ near $\partial\mM_R$, by a simple homotopy in the interior of $\mM_R$ and a simplified version  of the  analytic   Riemann-Roch property proved in \cite[Theorem~2.4]{MR1870658}, one obtains from Proposition \ref{prop:bdop_inverti}  the following proposition.

\begin{prop} \label{prop:riemroch}
For any $R>0$, there exists $\varepsilon_0>0$ such that if   $0<\varepsilon\leqslant \varepsilon_0$,
 then
\begin{equation}\label{riemroch}
 \ind \left(D^{\mM_R}_{\varepsilon,R,+},
P_{\geqslant 0,\varepsilon,R,+}\right)=\ind\left (D^M_+\right).
\end{equation}
\end{prop}

 Let $dv_{\varepsilon,R}$ denote the volume element
of $(\mM_R,{g}^{T\mM}_{\varepsilon,R})$, and
   $\|\cdot\|_{\varepsilon,R}$  denote the corresponding $L^2$-norm.
From (\ref{asindex}) and Propositions \ref{prop:bdop_inverti} and \ref{prop:riemroch}, one sees that in order to prove Theorem \ref{t0.1}, one need only to prove the following result.

\begin{thm}\label{t3.4}

Under the assumptions of Theorem \ref{t0.1},
there exists $R_0>0$ such that for any $R\geqslant R_0$, there exist $c_1>0$ and $\varepsilon_1>0$ such that
for any $0<\varepsilon\leqslant \varepsilon_1$ and smooth section $s\in\Gamma(\mM_R, S_{\varepsilon,R}(\mE)\widehat{\otimes}\Lambda(\mE^{\perp,*})
\widehat{\otimes}\Lambda(\mF^{*}))$ verifying $P_{\geqslant 0,\varepsilon,R}(s|_{\partial \mM_R})=0$, one has
\begin{align}\label{3.6}
\left\| D^{\mM_R}_{\varepsilon,R}s\right\|^2_{\varepsilon,R}\geqslant
c_1\left(\sum_{i=1}^{\dim \mM}\left\|\widehat\nabla^{\varepsilon,R}_{h_i}s\right\|^2_{\varepsilon,R}+\frac{1}{\varepsilon^2}\|s\|^2_{\varepsilon,R}\right).
\end{align}
\end{thm}

\section{Proof of Theorem \ref{t3.4}}\label{sec:proofmain}

This section is organized as follows. In Section \ref{s4.1}, we establish an estimating result near $\partial\mM_R$. In Section \ref{s4.2}, we establish two interior estimating results. In Section \ref{s4.3}, we complete the proof of Theorem \ref{t3.4}.

\subsection{The estimate near the boundary}\label{s4.1}

In this subsection, we prove the following estimating result near $\partial\mM_R$.

\begin{prop}\label{prop:bdestimate}
For any   $R>0$, there exist $c_2>0$ and $\varepsilon_2>0$ such that for any $0<\varepsilon\leqslant \varepsilon_2$ and any  $s\in \Gamma(\mM_R,S_{\varepsilon,R}(\mE)\widehat{\otimes}
\Lambda(\mE^{\perp,*})\widehat{\otimes}\Lambda(\mF^{*}))$ verifying ${\rm Supp}(s)\subseteq\mM_R\setminus \mM_{\frac{3}{4}R}$ and $ P_{\geqslant 0,\varepsilon,R}(s|_{\partial \mM_R})=0$,
 one has
\begin{align}\label{est_nearbd}
\left\|D^{\mM_R}_{\varepsilon,R}s\right\|^2_{\varepsilon,R}
\geqslant c_2\sum_{i=1}^{\dim \mM}\left\|\widehat\nabla _{h_i}s\right\|^2_{\varepsilon,R}+\frac{1}{2\,\varepsilon^2}\|s\|^2_{\varepsilon,R}.
\end{align}
\end{prop}

\begin{proof} Since $\psi(\frac{\rho}{R})=1$ on $\mM_R\setminus\mM_{\frac{3}{4}R}$,
by (\ref{subdirac_epsilonR}) and (\ref{sbdirac_deform}),
one has on $\mM_R\setminus\mM_{\frac{3}{4}R}$ that
\begin{align}\label{3.7}
D^{\mM_R}_{\varepsilon,R}= D^{\mM_R} + \frac{\widehat c(\rho\sigma)}{\varepsilon R}.
\end{align}

We now proceed as in the  proof of \cite[Proposition 2.4]{TZ}.
By Green's formula (cf. \cite[(2.28)]{MR1252030}), one deduces that
\begin{align}\label{integrationbyparts}
\left\|D^{\mM_R}_{\varepsilon,R} s\right\|^2_{\varepsilon,R}
&
=\int_{\mM_R}
\left\langle s,\left (D^{\mM_R}_{\varepsilon,R}\right)^2 s\right\rangle
{dv}_{\varepsilon,R}
\notag\\
&\qquad
+\int_{\partial\mM_R}\left\langle s,
c\left(e_{\dim \mM}\right)D^{\mM_R}_{\varepsilon,R} s\right\rangle
dv_{\partial\mM_R}
\end{align}
for any $s\in \Gamma(\mM_R,S_{\varepsilon,R}(\mE)\widehat{\otimes}
\Lambda(\mE^{\perp,*})\widehat{\otimes}\Lambda(\mF^{*}))$ with ${\rm Supp}(s)\subseteq\mM_R\setminus \mM_{\frac{3}{4}R}$.

By (\ref{3.7}), one has
\begin{align}\label{(sbdirac_deform)2}
\left(D^{\mM_R}_{\varepsilon,R}\right)^2
=\left(D^{\mM_R} \right)^2+\frac{1}{\varepsilon R}\left[ D^{\mM_R} ,\widehat c(\rho\sigma)\right]+\frac{\rho^2}{\varepsilon^2R^2}.
\end{align}

Following \cite[(2.26) and (2.27)]{MR1252030}, one verifies that on $\partial\mM_R$\,,
\begin{align}\label{bd_relation}
c(e_{\dim \mM})D^{\mM_R}_{\varepsilon,R}
=-\widehat{\nabla}_{e_{\dim\mM}}
-D^{\partial\mM_R}_{\varepsilon}
+\frac{1}{2}\sum_{i=1}^{\dim\mM-1}\pi_{ii}\ .
\end{align}

By using the Lichnerowicz formula for $(D^{\mM_R}  )^2$ and proceeding as as
 \cite[(2.10) and (2.11)]{TZ}, one gets for section $s$ with ${\rm Supp}(s)\subseteq\mM_R\setminus \mM_{\frac{3}{4}R}$ that
\begin{multline}\label{3.8}
\int_{\mM_R}
\left\langle s,\left (D^{\mM_R}\right)^2 s\right\rangle
{dv}_{\varepsilon,R} - \int_{\partial\mM_R}\left\langle s,\widehat{\nabla}_{e_{\dim\mM}}\right\rangle
dv_{\partial\mM_R}
\\
=\sum_{i=1}^{\dim\mM}\left\|\widehat \nabla_{h_i}s\right\|^2 +O_R(1)\,\|s\|^2.
\end{multline}

By Proposition \ref{prop:bdop_inverti}   and proceeding as in \cite[(2.21)]{TZ}, one sees that when $\varepsilon>0$ is small enough,   for any smooth section $s$ verifying   $ P_{\geqslant 0,\varepsilon,R}(s|_{\partial \mM_R})=0$, one has that
\begin{align}\label{3.9}
\int_{\partial\mM_R}\left\langle s, -D^{\partial\mM_R}_{\varepsilon} s\right\rangle dv_{\partial\mM_R}
+\frac{1}{2}\int_{\partial\mM_R}\left\langle s,\sum_{i=1}^{\dim\mM-1}\pi_{ii}\,s\right\rangle  dv_{\partial\mM_R}
\geqslant 0.
\end{align}

Since $[D^{\mM_R},\widehat c(\rho\sigma)]$ is of zeroth order, and $\frac{\rho}{R}\geqslant \frac{3}{4}$ on $\mM_R\setminus\mM_{\frac{3}{4}R}$, from (\ref{integrationbyparts})-(\ref{3.9}), one gets (\ref{est_nearbd}).
\end{proof}

\subsection{The interior estimates}\label{s4.2}

From now on, we  assume that there exists $\delta>0$ such that
\begin{equation}\label{posi_scala}
k^{TM}\geqslant \delta\ \text{over}\ M.
\end{equation}

We prove two interior estimating results. The first one is as follows.

\begin{prop}\label{prop:intestimate}
There exists $R_1>0$ such that for any $R\geqslant R_1$,
there exist $c_3>0$ and $\varepsilon_3>0$ such that for any $0<\varepsilon\leqslant\varepsilon_3$, one has
\begin{align}\label{inter_esti}
\left\|D^{\mM_R}_{\varepsilon,R}
s\right\|_{\varepsilon,R}^2
\geqslant  c_3\sum_{i=1}^{\dim\mM} \left\|\widehat\nabla^{\varepsilon,R}_{h_i}s\right\|^2_{\varepsilon,R}
+\frac{\delta}{9\,\varepsilon^2}\|s\|_{\varepsilon,R}^2,
\end{align}
for any $s\in \Gamma(\mM_R,S_{\varepsilon,R}(\mE)\widehat{\otimes}
\Lambda(\mE^{\perp,*})\widehat{\otimes}\Lambda(\mF^{*}))$
supported in $\mM_{\frac{2}{3}R}$\,.
\end{prop}
\begin{proof} Recall that one has $\psi(\frac{\rho}{R})=1$ on $\mM_{\frac{2}{3}R}$.
From (\ref{subdirac_epsilonR}), one has on $\mM_{\frac{2}{3}R}$ that
\begin{align}\label{3.13}
\left(D^{\mM_R}_{\varepsilon,R}\right)^2
=\left(D_\varepsilon^{\mM_R} \right)^2+\frac{1}{\varepsilon R}\left[ D_\varepsilon^{\mM_R} ,\widehat c(\rho\sigma)\right]+\frac{\rho^2}{\varepsilon^2R^2}.
\end{align}

In view of  (\ref{subdirac_epsilonR}), (\ref{[conn,hatc]}) and (\ref{3.13}), we compute
\begin{align}\label{3.10}
\left[ D_\varepsilon^{\mM_R},\widehat c(\rho\sigma)\right]=\sum_{i=1}^{\dim M}c_{\varepsilon,R}\left(h_i\right)\widehat c\left(  \nabla^{\mE^\perp}_{h_i}(\rho\sigma)\right)
+ \sum_{j=\dim M+1}^{\dim \mM}c\left(h_j\right)\widehat c\left( \nabla^{\mE^\perp}_{h_j}(\rho\sigma)\right).
\end{align}

From Lemma \ref{lem:esti} and (\ref{3.10}), one finds
\begin{align}\label{3.11}
\left[ D_\varepsilon^{\mM_R},\widehat c(\rho\sigma)\right]= \frac{O(1)}{\varepsilon}+O_R(1).
\end{align}

From Lemma \ref{t3.1}, (\ref{Lichnerowicz}), (\ref{scalarcurv}) and (\ref{posi_scala}), one finds (compare with \cite[(1.71) and (2.28)]{Zhangfoliation})
\begin{align}\label{inter_intgr}
 \left(D_\varepsilon^{\mM_R}\right)^2
\geqslant -\Delta^{\mM_R}_{\varepsilon,R} + \frac{\delta}{4\,\varepsilon^2}  + \frac{ O_R(1)}{\varepsilon},
\end{align}
where $ -\Delta^{\mM_R}_{\varepsilon,R}\geqslant 0$ is the corresponding Bochner Laplacian.

From (\ref{3.13}), (\ref{3.11}) and (\ref{inter_intgr}), one completes the proof of Proposition \ref{prop:intestimate} easily.
\end{proof}

The second interior estimating result can be stated as follows.

\begin{prop}\label{t4.4}
There exists $R_2>0$ such that
for any $R\geqslant R_2$,
there exist $c_4>0$ and $\varepsilon_4>0$ such that for any $0<\varepsilon\leqslant\varepsilon_4$, one has
\begin{align}\label{3.14}
\left\|D^{\mM_R}_{\varepsilon,R}
s\right\|_{\varepsilon,R}^2
\geqslant  c_4\sum_{i=1}^{\dim\mM} \left\|\widehat\nabla^{\varepsilon,R}_{h_i}s\right\|^2_{\varepsilon,R}
+\frac{1}{11\,\varepsilon^2}\|s\|_{\varepsilon,R}^2,
\end{align}
for any $s\in \Gamma(\mM_R,S_{\varepsilon,R}(\mE)\widehat{\otimes}
\Lambda(\mE^{\perp,*})\widehat{\otimes}\Lambda(\mF^{*}))$
supported in $\mM_{\frac{7}{8}R}\setminus\mM_{\frac{1}{3}R}$\,.
\end{prop}
\begin{proof} On $\mM_{\frac{7}{8}R}\setminus\mM_{\frac{1}{3}R}$,  one has
\begin{align}\label{3.15}
\frac{\rho}{R}\geqslant \frac{1}{3}.
\end{align}
Also, one finds that (\ref{3.11}) still holds on $\mM_{\frac{7}{8}R}\setminus\mM_{\frac{1}{3}R}$.

From (\ref{3.13}), (\ref{3.11}) and (\ref{3.15}), one finds that there exists $R_3>0$ such that for any $R\geqslant R_3$, 
 when $\varepsilon>0$ is small enough, one has
\begin{align}\label{3.16}
\left\|D^{\mM_R}_{\varepsilon,R}
s\right\|_{\varepsilon,R}^2
\geqslant  \left\|D^{\mM_R}_{\varepsilon}s\right\|^2_{\varepsilon,R}
+\frac{1}{10\,\varepsilon^2}\|s\|_{\varepsilon,R}^2,
\end{align}
for any $s\in \Gamma(\mM_R,S_{\varepsilon,R}(\mE)\widehat{\otimes}
\Lambda(\mE^{\perp,*})\widehat{\otimes}\Lambda(\mF^{*}))$
supported in $\mM_{\frac{7}{8}R}\setminus\mM_{\frac{1}{3}R}$\,.

Formula (\ref{3.14}) follows easily from (\ref{3.16}).
\end{proof}

\subsection{Proof of Theorem \ref{t3.4}}\label{s4.3}

Let $\beta_1: [0,1]\rightarrow [0,1]$ be a smooth function such that $\beta_1(t)=0$ for $0 \leqslant t \leqslant\frac{3}{4}+\frac{1}{100}$, while $\beta_1(t)=1$ for $\frac{3}{4}+\frac{1}{50}\leqslant t \leqslant 1$. Let $\beta_2: [0,1]\rightarrow [0,1]$ be a smooth function such that $\beta_2(t)=1$ for $0 \leqslant t \leqslant\frac{1}{2}+\frac{1}{100}$, while $\beta_2(t)=0$ for $\frac{2}{3}-\frac{1}{50}\leqslant t \leqslant 1$.

Inspired by \cite[pp.~115-116]{MR1188532}, let $\alpha_1$, $\alpha_2$ and $\alpha_3$ be the smooth functions on $\mM_R$ defined by
\begin{align}\label{3.17}
\alpha_1=\frac{ \beta_1\left(\frac{\rho}{R}\right)}{\sqrt{\beta_1\left(\frac{\rho}{R}\right)^2+\beta_2\left(\frac{\rho}{R}\right)^2+\left(1-\beta_1\left(\frac{\rho}{R}\right)-\beta_2\left(\frac{\rho}{R}\right)\right)^2}}\ ,
\end{align}
\begin{align}\label{3.18}
\alpha_2=\frac{ \beta_2\left(\frac{\rho}{R}\right)}{\sqrt{\beta_1\left(\frac{\rho}{R}\right)^2+\beta_2\left(\frac{\rho}{R}\right)^2+\left(1-\beta_1\left(\frac{\rho}{R}\right)-\beta_2\left(\frac{\rho}{R}\right)\right)^2}}
\end{align}
and
\begin{align}\label{3.19}
\alpha_3=\frac{1- \beta_1\left(\frac{\rho}{R}\right)- \beta_2\left(\frac{\rho}{R}\right) }{\sqrt{\beta_1\left(\frac{\rho}{R}\right)^2+\beta_2\left(\frac{\rho}{R}\right)^2+\left(1-\beta_1\left(\frac{\rho}{R}\right)-\beta_2\left(\frac{\rho}{R}\right)\right)^2}}.
\end{align}
Then $\alpha_1^2+\alpha^2_2+\alpha_3^2=1$ on $\mM_R$. Thus,    for any $s\in \Gamma(\mM_R,S_{\varepsilon,R}(\mE)\widehat{\otimes}
\Lambda(\mE^{\perp,*})\widehat{\otimes}\Lambda(\mF^{*}))
$, one has
\begin{multline}\label{cutoff_alpha}
\left\|D^{\mM_R}_{\varepsilon,R}s\right\|_{\varepsilon,R}^{2}
=\sum_{i=1}^3 \left\|\alpha_i D^{\mM_R}_{\varepsilon,R}s\right\|_{\varepsilon,R}^{2}
\\
\geqslant\frac{1}{2} \sum_{i=1}^3\left\|D^{\mM_R}_{\varepsilon,R}\left(\alpha_is\right)\right\|_{\varepsilon,R}^{2}
-\sum_{i=1}^3\left\|{c}_{\varepsilon,R}\left(d\alpha_i\right) s\right\|_{\varepsilon,R}^2\ ,
\end{multline}
where we identify  each $d\alpha_i$ ($i=1,\,2,\,3$) with the gradient of $\alpha_i$.

From Lemma \ref{lem:esti} and (\ref{3.17})-(\ref{3.19}), one finds (compare with \cite[(2.47)]{Zhangfoliation})
\begin{align}\label{3.20}
\sum_{i=1}^3\left|{c}_{\varepsilon,R}\left(d\alpha_i\right) \right|_{\varepsilon,R} =\frac{O(1)}{\varepsilon R}+O_R(1).
\end{align}

Clearly, ${\rm Supp}(\alpha_1s)\subseteq \mM_R\setminus \mM_{\frac{3}{4}R}$,
${\rm Supp}(\alpha_2s)\subseteq   \mM_{\frac{2}{3}R}$ and ${\rm Supp}(\alpha_3s)\subseteq \mM_{\frac{7}{8}R}\setminus \mM_{\frac{1}{3}R}$.

From Propositions \ref{prop:bdestimate}-\ref{t4.4} and formulas (\ref{cutoff_alpha}) and (\ref{3.20}),
one completes the proof of Theorem \ref{t3.4} easily (compare with \cite[pp.~115-117]{MR1188532}).

The proof of Theorem \ref{t0.1} is thus also completed.


\end{document}